\documentclass{article}

\usepackage[english]{babel}
\usepackage{amsthm}
\usepackage{graphicx}
\usepackage{authblk}
\usepackage{caption}
\usepackage{subcaption}
\usepackage{amssymb}
\usepackage{amsmath}
\usepackage{xfrac}
\usepackage{color}
\usepackage{hyperref}
\usepackage[normalem]{ulem}
\usepackage{relsize}
\usepackage[document]{ragged2e}

\newtheorem*{thm*}{Theorem}
\newtheorem{thm}{Theorem}

\newtheorem{obs}[thm]{Observation}
\newtheorem{lemma}[thm]{Lemma}
\newtheorem{cor}[thm]{Corollary}
\newtheorem{rem}[thm]{Remark}
\newtheorem{conj}[thm]{Conjecture}
\numberwithin{thm}{section}
\begin{document}
\date{\today}
\author[1]{Ilan Karpas \thanks{Department of Mathematics, Hebrew University of Jerusalem.\\Email: ilan.karpas@mail.huji.ac.il}}

\title{Two Results on Union-Closed Families}

\maketitle

\begin{abstract}
 We show that there is some absolute constant $c>0$, such that for any union-closed family $\mathcal{F} \subseteq 2^{[n]}$, if \mbox{$|\mathcal{F}| \geq (\frac{1}{2}-c)2^n$}, then there is some element $i \in [n]$ that appears in at least half of the sets of $\mathcal{F}$. We also show that for any union-closed family $\mathcal{F} \subseteq 2^{[n]}$, the number of sets which are not in $\mathcal{F}$ that cover a set in $\mathcal{F}$ is at most $2^{n-1}$, and provide examples where the inequality is tight.
\end{abstract}
\section{Introduction}
The objects of study in this paper are \textit{union-closed} families. A family \newline $\mathcal{F} \subseteq 2^{[n]}$ is called \textit{union-closed}, if for every two sets $A,B \in \mathcal{F}$,\mbox{$A \cup B \in \mathcal{F}$}.  There is a wide range of literature concerning various properties of union-closed families.
For instance, Alekseev \cite{Alekseev} approximated the number of union-closed families of subsets of $[n]$, Kleitman \cite{Kleitman} gave an upper-bound for the number of basis sets for such families, and Reimer \cite{Reimer} found a tight lower-bound for the average size of a set inside a union-closed family containing $m$ sets. \vspace{5mm}

However, if you stop a combinatorialist on the street, and ask him what is the best-known conjecture regarding union-closed sets, most probably he will respond
"Frankl's conjecture", or simply "The union-closed set conjecture". This conjecture was made by Peter Frankl in the late $1970$'s. The conjecture asserts that for every finite union-closed family which contains a non-empty set, there is some element that belongs to at least half of its members. Formally, for a family $\mathcal{F} \subseteq 2^{[n]}$ and $i \in [n]$, we write $\mathcal{F}_i:=\{A \in \mathcal{F}| i \in A\}$. If $\frac{|\mathcal{F}_i|}{|\mathcal{F}|}\geq \frac{1}{2}$ we say that $i$ is \textit{abundant in} $\mathcal{F}$. If, on the otherhand, $\frac{|\mathcal{F}_i|}{|\mathcal{F}|}\leq \frac{1}{2}$  we say that $i$ is \textit{rare in} $\mathcal{F}$. Frankl's conjecture can be stated as follows:

\begin{conj} 
Let $\mathcal{F} \subseteq 2^{[n]}$ be a union-closed family, $\mathcal{F}\neq \{\emptyset\}$. Then there is some element $i \in [n]$ that is abundant in $\mathcal{F}$.
\end{conj}

 There have been various partial results regarding the conjecture. Let us first recommend the survey paper of  Bruhn and Schaudt \cite{Bruhn} for details.
 Vu\v{c}kovi\'{c} and \v{Z}ivkoi\'{c} \cite{Vuckovic} showed that the conjecture is shown to be true for any union-closed family $\mathcal{F}$ with a universe of at most $12$ elements, improving  on results in \cite{Bosnjak, Markovic, Morris, Gao, Poonen}. Lo Faro \cite{Lo Faro}, and later independently Roberts and Simpson \cite{Roberts}, showed that if $\mathcal{F}$ is a counterexample to the conjecture, and the universe of $\mathcal{F}$ is of size $q$, then $|\mathcal{F}| \geq 4q-1$. Falgas-Ravry \cite{Falgas-Ravry} showed that the conjecture holds for  any separating union-closed family $\mathcal{F}$ 
with a universe of $n$ elements with at most $2n$ elements (separating here means that no two distinct elements $i,j \in [n]$ appear in exactly the same sets in $\mathcal{F}$). This was slightly improved by Ma{\ss}berg \cite{Massberg}, from $2n$ to $2(n+\frac{n}{\log_2n -\log\log_2n})$. Interestingly, Hu \cite{Hu} proved that if this bound can be improved to $(2+c)n$ for some constant $c>0$ , this already implies that \textbf{any} union-closed family $\mathcal{F}$ has an element appearing in at least $\frac{c-2}{2(c-1)}|\mathcal{F}|$ sets in $\mathcal{F}$. At the moment, all that is known is that each union-closed family $\mathcal{F}$ has an element occuring in at least $\Omega(\frac{|\mathcal{F}|}{\log_2|\mathcal{F}|})$ sets in $\mathcal{F}$ \cite{Knill, Wojcik}. A Polymath project was dedicated to try to prove the conjecture, but without success (so far) \cite{Polymath}. \vspace{5mm}

Of more relevance to this paper, are results proving that the conjecture holds for union-closed families $\mathcal{F}$ with many sets, compared to the size of the universe. Cz\'{e}dli \cite{Czedli1} proved that for any union-closed family $\mathcal{F} \subset 2^{[n]}$, where
$|\mathcal{F}|\geq 2^n - 2^{n/2}$, the conjecture holds. This was significantly improved by Balla, B\'{o}llobas and Eccles \cite{Balla} to all union-closed families of subsets of $[n]$ of size at least $\frac{2}{3}2^n$, and then further improved by Eccles \cite{Eccles} to $(\frac{2}{3}-\frac{1}{104})2^n$. Our first theorem, is that the bound can be improved to $\frac{1}{2}2^n$:

\begin{thm} \label{abundant}
Let $\mathcal{F} \subset 2^{[n]}$ be a union-closed family, where $|\mathcal{F}|\geq 2^{n-1}$. Then there is some element $i \in [n]$, so that
$|\mathcal{F}_i| \geq \frac{1}{2}|\mathcal{F}|$.
\end{thm}

We use basic Boolean-Analysis techniques to prove this result. To the best of our knowledge, Boolean-Analysis has not been used before to tackle Frankl's conjecture.

The second main theorem of this paper, regards the maximum number  of sets in  the \textit{upper-shadow} of a union-closed family, which are not in the family itself.

For a subset $A \subseteq [n]$, The upper-shadow of $A$ with respect to $n$, denoted as $\partial^+A$, is the following:

$$\partial^+A = \{A \cup \{i\}| i \in [n] \setminus A\}$$.

 For a family of sets $\mathcal{F} \subset 2^{[n]}$, the upper-shadow of the family is just the union of the upper-shadow of all sets in $\mathcal{F}$:

$$\partial^+\mathcal{F}=\bigcup_{A \in \mathcal{F}}\partial^+A$$.

We show that the upper shadow of a union-closed family can not be too large:

\begin{thm}\label{upper-shadow}
Let $\mathcal{F} \subseteq 2^{[n]}$ be a union-closed family. Then

$$ |\partial^+ \mathcal{F}\setminus \mathcal{F}| \leq 2^{n-1}.$$ 
\end{thm}

Note that this is tight, for instance by considering the family  \mbox{$\mathcal{F}=\{A \subseteq [n]|1 \notin A\}$}.\\ \vspace{3mm}
Lastly, we use Theorem \ref{upper-shadow}, combined with more advanced Boolean-Analysis results, to improve slightly on Theorem \ref{abundant}. We show:

\begin{thm} \label{abundant-2}
Let $\mathcal{F} \subset 2^{[n]}$ be a union-closed family, where $|\mathcal{F}|\geq (\frac{1}{2}-c)2^n$. Then there is some element $i \in [n]$, so that
$|\mathcal{F}_i| \geq \frac{1}{2}|\mathcal{F}|$.
\end{thm}

Although Theorem \ref{abundant} is, of course, an immediate consequence of Theorem \ref{abundant-2}, nevertheless we have decided to include the proof of the former, both because the proof is simpler and because its proof can be extended naturally to a more general setting than union-closed families (see section $3$ for details).

The structure of the paper is as follows:

in section $2$, we provide the necessary definitions and tools from Boolean-Analysis, and prove some basic properties of union-closed families. In section $3$ we prove Theorem \ref{abundant}, and in section $4$ we prove Theorem \ref{upper-shadow}. In section $5$ we prove Theorem \ref{abundant-2}, and  finally, in section $6$, we discuss some implications and open problems stemming from this paper's results.

\section{Preliminaries}

\subsection{Boolean Analysis}
In most of this subsection we provide basic definitons and facts from Boolean-Analysis. Towards the end we cite two non-trivial theorems in Boolean-Analysis.

We identify subsets of $[n]$ with boolean strings of length $n$, by associating with each $S \subseteq [n]$ the string $x^S=x^S_1\dots x^S_n \in \{0,1\}^n$,
where $x^S_i=1$ iff $i \in S$. 

We define an inner product on boolean vectors $\langle . , . \rangle:\{0,1\}^n\times \{0,1\}^n \to \mathbb{Z}$. For $x=x_1\dots x_n$, $y=y_1\dots y_n$, their inner product is

 $$\langle x,y \rangle=\sum_{i=1}^n x_iy_i .$$

The space of boolean functions itself also has an inner product.
For functions $f,g:\{0,1\}^n \to \{-1,1\}$, their inner product is taken to be:

$$ \langle f, g \rangle = \mathbb{E}_{x \sim \{0,1\}^n} f(x)g(x).$$

Notice that this number is always between $-1$ and $1$. The \textit{distance} between $f$ and $g$ is then defined as

$$dist(f,g)=\frac{1}{2}(1-\langle f, g \rangle).$$

The distance between two classes of boolean functions on $n$-coordinates $\mathcal{A}$ and $\mathcal{B}$ is the minimal distance between a function in $\mathcal{A}$ and a function in $\mathcal{B}$.
 
A special role is played by the so called \textit{character functions}. For every $S \subseteq [n]$ there exists a unique character function of $S$,  $\chi_S:\{0,1\}^n \to \{-1,1\}$, which is defined thus:

$$\chi_S(x)=(-1)^{\langle x, x_S \rangle}.$$

In the special case that $|S|=1$, the function $\chi_S(x)$ is known as a \textit{dictator}. Similarly, in this case, $-\chi_S(x)$ is called an \textit{anti-dictator}. The set of all character functions is an orthonormal basis (with the inner product we have defined) for the space of boolean functions $f:\{0,1\}^n \to \{-1,1\}$. For a boolean function $f$, we define its \textit{fourier coefficient} for $S \subseteq [n]$ to be 

$$\hat{f}(S)=\langle f,\chi_s\rangle.$$

Notice that this is the coefficient of $\chi_s$ in the unique representation of $f$ as a linear combination of the characters. We call all fourier-coefficients of sets of size $k$ the \textit{level $k$ coefficients of $f$}. The sum of squares of all level $k$ coefficients of $f$ is called the \textit{level-$k$ weight of $f$}, and is denoted as

$$W^k(f):=\sum_{|S|=k}\hat{f}(S)^2.$$

At the heart of boolean analysis lays the following identity, known as Parseval's identity:
\begin{lemma} \label{Parseval}
For any $f:\{0,1\}^n \to \{-1,1\}$ 

\begin{equation}
\sum_{k=0}^n W^k(f)=1
\end{equation}
\end{lemma}

For any coordinate $i$, we define the $i$th positive (negative) influence, $I_i^+(f)$ thus:

\begin{flalign*}
&I_i^+(f)= \mathbb{P}_{S \sim [n] \setminus \{i\}}(f(x^S)= -1 \wedge f(x^{S \cup \{i\}})= 1)\\
(&I_i^-(f)=\mathbb{P}_{S \sim [n] \setminus \{i\}}(f(x^S)=1 \wedge f(x^{S \cup \{i\}})=-1))
\end{flalign*}

In other words, if we partition $\{0,1\}^n$ to $2^{n-1}$ pairs of the form $(x, x\oplus e_i)$ where $x_i=0$, then the $i$th positive (negative) influence is the fraction of all such pairs for which $f(x)=1$ and $f(x\oplus e_i)=-1$ ($f(x)=-1$ and \mbox{$f(x \oplus e_i)=1$)}. We then define the \textit{$i$th influence}, $I_i(f)$, to be the sum of these two:

$$I_i(f)=I_i^-(f)+I_i^-(f).$$

The \textit{positive (negative)} influence of $f$ is:

$$I^+(f)=\sum_{i=1}^n I_i^+(f)\quad (I^-(f)=\sum_{i=1}^n I_i^-(f)),$$

and the influence of $f$ is simply the sum of the positive and the negative influences:

$$I(f)=I^+(f)+I^-(f).$$

The reader new to the subject might want to get some feel for the definitions, by proving to himself the following observation:
\begin{obs} \label{levels}
Let $f: \{0,1\}^n \to \{-1,1\}$ be a boolean function. Then:
\begin{enumerate}
\item $\hat{f}(\emptyset)=1-2^{1-n}|f^{-1}(-1)|$.
\item For every $i \in [n]$, $\hat{f}(i)=I_i^+(f)-I_i^-(f)$.
\end{enumerate}
\end{obs}

The influence of a function $f$ has a nice analytic expression, whose proof can be found in any standard introduction to the subject (see, e.g., \cite{O'Donnell}).
\begin{equation} \label{influence}
I(f)=\sum_{i=0}^n W^i.
\end{equation}

In this paper, what we shall actually use later on is an analytic expression that provides a lower bound for the influence:

\begin{cor} \label{cor}

For any $f:\{0,1\}^n \to \{-1,1\}$, and any $k\in [n]$:

$$I(f)\geq k -\sum_{i=0}^{k-1} (k-i)W^i(f)$$.
\end{cor}

\begin{proof}
\begin{equation}
\begin{aligned}
I(f)={} & \sum_{i=0}^n iW^i(f) \geq \sum_{i=0}^{k-1} W^i(f) + \sum_{i=k}^n kW^i(f) =
        \\& =k\sum_{i=0}^n W^i(f) -\sum_{i=0}^{k-1} (k-i)W^i(f)=k- \sum_{i=0}^{k-1} (k-i)W^i(f)
\end{aligned}
\end{equation}
where the last equality uses Parseval's identity.
\end{proof}
We end this subsection by stating two important theorems in this subject. The first one is a famous result by Freidgut, Kalai and Naor \cite{FKN}, known in the field as FKN theorem. The theorem states that if a boolean function has almost all of its weight on level-$1$, then this function is close (distance here being as we have defined above) to some dictator or anti-dictator:

\begin{thm} \label{FKN} \textit{(FKN)} \cite{FKN} 
There is some absolute constant $C_1$, so that the following holds. For any boolean function $f:\{0,1\}^n \to \{-1,1\}$, if \mbox{$W^1(f)\geq 1-\delta$}, then there is some $i \in [n]$, for which either $dist(f,\chi_i)<C_1\delta$ or \mbox{$dist(f,-\chi_i)<C_1\delta$}.
\end{thm}

Several years after this theorem was discovered, Kindler and Safra \cite{Kindler-Safra} managed to show an analog for higher (constant) weights.
Specifically of interest to us is the level-$2$ concentration case.

\begin{thm} \label{Kindler-Safra}
There is some absolute constant $C_2>0$, so that the following holds. For any $n \in \mathbb{N}$, denote by $\mathcal{A}_n$ the class of all functions either of the form $\pm \chi_{i,j}$ for some distinct $i,j \in [n]$, or of the form \mbox{$\pm \frac{1}{2}(\chi_{i,j}+\chi_{j,k}+\chi_{k,l}-\chi_{i,l})$}, for some distinct $i,j,k,l \in [n]$. Let $f:\{0,1\}^n \to \{-1,1\}$ be a boolean function, so that 
$W^2(f)\geq 1-\delta$. Then $dist(f,\mathcal{A}_n)<C_2\delta$.
\end{thm}

Notice that, in both theorems, all functions which are used to approximate $f$ are \textit{balanced function}. That is, They have value $-1$ on exactly half of the points in $\{0,1\}^n$, or equivalently, their zero-level fourier coefficient is $0$.

\subsection{Union-closed and Simply-rooted Families}
\begin{rem}
For a set $A$ and an element $i \in A$, we denote by $[i,A]$ the family
of all sets containing element $i$ and contained in $A$. We often, by abuse of notation, write $i$ when we in fact mean the singleton $\{i\}$. The meaning should be clear from context.
\end{rem}

 Recall that a family $\mathcal{G} \in 2^{[n]}$ is called \textit{union-closed}, if for every two sets \mbox{$A,B \in \mathcal{G}$} $A \cup B \in \mathcal{G}$.
A family  $\mathcal{F} \subseteq 2^{[n]}$ is called \textit{simply-rooted}, if for every $A \in \mathcal{F}$, there
 is some $i \in A$, so that $[i,A] \subseteq \mathcal{F}$. We say in this case that $A$ is rooted in $i$.

\begin{lemma} \label{dual}
 $\mathcal{G}$ is union-closed iff $\mathcal{F}:=2^{[n]} \setminus \mathcal{G}$ is simply-rooted. 
\end{lemma}
\begin{proof}
Assume $\mathcal{G}$ is union-closed, and assume by contradiction $\mathcal{F}$ is not simply-rooted. That is, there exists some $A \in \mathcal{F}$ with the following property. For every $i \in A$ there is some set $A_i\in \mathcal{G}$, where $i \in A_i \subseteq A$. Notice that \mbox{$\bigcup_{i \in A}A_i =A$} by definition, but since $\mathcal{G}$ is union-closed, also $\bigcup_{i \in A}A_i \in \mathcal{G}$. This means that $A \in \mathcal{G}$, which is a contradiction.\\
\vspace{3mm}
For the other direction, assume that $\mathcal{F}$ is simply-rooted, and assume by contradiction $\mathcal{G}$ is not union-closed. So there are two sets $A,B \in \mathcal{G}$, with \mbox{$A \cup B \in \mathcal{F}$}. This means that there is some \mbox{$i \in A \cup B$} with \mbox{$[i,A\cup B] \subseteq \mathcal{F}$}, so $i$ is in $A$ or in $B$. Assume without loss of generality that \mbox{$i \in A$}. Then \mbox{$A \in [i, A\cup B]\subseteq \mathcal{F}$}. But this is a contradiction, since we took $A \in \mathcal{G}$.
\end{proof}

We have defined in the introduction the upper shadow of a set. We define, in similar manner, the \textit{lower shadow}. Let $A \subseteq [n]$. Then the \textit{lower shadow of A}, denoted by $\partial^-A$, is defined as follows:

$$\partial^- A=\{A\setminus i|i \in A\}.$$

For a family of sets $\mathcal{F} \subseteq 2^{[n]}$, the lower shadow of $\mathcal{F}$ is simply the union of the lower shadow of all sets inside $\mathcal{F}$:

$$\partial^-\mathcal{F}=\bigcup_{A \in \mathcal{F}}\partial^-A.$$

With this definition at hand, we prove:
\begin{lemma} \label{shadow}
Let $\mathcal{F}$ be a simply-rooted family. Then for every $A \in \mathcal{F}$, 

\[\partial^-A \setminus \mathcal{F}= \left\{\begin{array}{lr}
       \{A\setminus {i}\}, & \text{if $A$ is rooted in $i$ and in no other element.}\\ 
  \emptyset,  \text{otherwise.}
        \end{array}\right\}
\]
\end{lemma}

\begin{proof}

Let $A \in \mathcal{F}$, with $A$ rooted in element $i$. Then $[i,A] \subseteq \mathcal{F}$, which means that for every $j\in A$ with $j \neq i$, $A \setminus j \in [i,A]$, and thus $A\setminus j \in \mathcal{F}$.
If $A$ is rooted in another element $j$, then by the same token $A \setminus i \in [j,A] \subseteq  \mathcal{F}$. So in this case $\partial^-A \setminus \mathcal{F}=\emptyset$.\\ \vspace{3mm}
If, on the other hand, $A$ is rooted only in element $i$, then for every element $j \in A \setminus i$, there is some set $A_j\in \mathcal{G}$, satisfying \mbox{$j \in A_j \subseteq A\setminus i$}. Taking the union of them, and recalling that $\mathcal{G}$ is union-closed, we obtain \mbox{$A \setminus i = \bigcup_{j \in A \setminus i}A_j \in \mathcal{G}$}, proving the lemma. 
\end{proof}

\section{Frankl's conjecture for large families}
Due to Lemma \ref{dual}, Theorem \ref{abundant} can be stated in terms of simply-rooted families, rather than union-closed sets. The key observation here, is that every element $i \in [n]$ appears in exactly half the sets in $2^{[n]}$. Thus, an element $i \in [n]$ is abundant in family $\mathcal{F}$ iff it is rare in $2^{[n]} \setminus \mathcal{F}$:

\begin{thm*} \textbf{\ref{abundant}*}
Let $\mathcal{F} \subseteq 2^{[n]}$ be simply-rooted, $|\mathcal{F}| \leq 2^{n-1}$. Then there is some element $i \in [n]$ such that $i$ is rare in $\mathcal{F}$.
\end{thm*}

\begin{proof}
Let $\mathcal{F} \subseteq 2^{[n]}$ be a simply-rooted family of size $|\mathcal{F}|=(\frac{1}{2}-\delta)2^n$, where $\delta \geq 0$. Assume by contradiction that $|\mathcal{F}_i|>\frac{1}{2}|\mathcal{F}|$ for every $i \in [n]$. Identify $\mathcal{F}$ with a function $f:\{0,1\}^n \to \{-1,1\}$ as in section $2.1$. 
Our goal is to find both a lower and an upper bound for the positive influence $I^+(f)$, and then show that the lower bound is in fact larger than the upper bound. This is of course impossible.

For the upper bound, observe that, by Lemma \ref{shadow}, for any $x \in \{0,1\}^n$ such that $f(x)=-1$, there is at most one $i \in [n]$ for which $x_i=1$ and \mbox{$f(x\oplus e_i)=1$}. Thus, each such $x$ contributes at most $2^{1-n}$ to the positive influence. Since \mbox{$|f^{-1}(-1)|=(\frac{1}{2}-\delta)2^n$}, we  deduce the upper bound: 

\begin{equation} \label{upper-bound}
I^+(f)\leq 1-2\delta.
\end{equation}

What about the lower bound?

By Lemma \ref{levels}, we have:
\begin{equation} \label{level-0}
\hat{f}(\emptyset)=1-2(\frac{1}{2}-\delta)=2\delta.
\end{equation}

Furthermore, for any $i \in [n]$:

\begin{equation} \label{level-1}
\hat{f}(i)=I_i^+(f)-I_i^-(f)>0.
\end{equation}

Indeed, (\ref{level-1}) is equivalent to our assumption that $|\mathcal{F}_i|>\frac{1}{2}|\mathcal{F}|$ for every $i \in [n]$.

Because all level $1$ fourier coefficients are strictly between $0$ and $1$, then \mbox{$\hat{f}(i)>\hat{f}(i)^2$} for all $i \in [n]$.
 Summing over all $i$'s, we obtain:

\begin{equation*}I^+(f)-I^-(f)=\sum_{i=1}^n \hat{f}(i) > \sum_{i=1}^n \hat{f}(i)^2.
\end{equation*}

Plugging the latter inequality to the one stated in Corollary \ref{cor}, for $k=2$, gives:

$$I^+(f)+I^-(f)\geq 2- \sum_{i=1}^n \hat{f}(i)^2 -2\hat{f}(\emptyset)^2>2-(I^+(f)-I^-(f))-8\delta^2,$$

or 

\begin{equation} \label{lower-bound}
I^+(f)>1-4\delta^2
\end{equation}

Finally, by combining (\ref{upper-bound}) and (\ref{lower-bound}), we see that

$$\delta>\frac{1}{2},$$

but this is absurd, since \mbox{$|\mathcal{F}|=(\frac{1}{2}-\delta)2^n$} can not be a negative number, and the theorem is proved.
\end{proof}

\begin{rem}
Notice that the above theorem also holds if instead of demanding that $\mathcal{F}$ is simply-rooted, we make the weaker demand that every set in $\mathcal{F}$ covers at most one set that is not in $\mathcal{F}$. Indeed, this is the only property of simply-rooted families used in the proof.
\end{rem}

\section{Upper Shadow of Union-closed families}

Let $\mathcal{G} \subseteq 2^{[n]}$ be a union-closed family, and let $\mathcal{F}:= 2^{[n]} \setminus \mathcal{G}$. By Lemma \ref{dual}, $\mathcal{F}$ is a simply-rooted family. By definiton of the upper-shadow, for any set $A$, $A \in \partial^+\mathcal{G} \setminus \mathcal{G}$ iff $A \in \mathcal{F}$, and there exists some $i \in A$ such that $A \setminus i \in \mathcal{G}$. However, from Lemma \ref{shadow}, this happens exactly when $A \in \mathcal{F}$ and is rooted in exactly one element. So Theorem \ref{upper-shadow}
can be equivalently stated like so:

\begin{thm*} \textit{\ref{upper-shadow}*}
Let $\mathcal{F}$ be a simply-rooted family of subsets of $[n]$. Then there are at most $2^{n-1}$ sets in $\mathcal{F}$ that are rooted in only one element.
\end{thm*}

\begin{proof}
 We identify $2^{[n]}$ with the set of vertices of the $n$-dimensional Hamming cube, $V(Q_n)$ in the obvious manner. We do this, since we shall use the following theorem of Kotlov from $2000$:

\begin{thm} (Kotlov, $2000$) \label{Kotlov} \cite{Kotlov}
Denote by $Q_n$ the $n$-dimensional Hamming cube, and let $V$ be a subset of the vertices, so that $|V|>2^{n-1}$. Then in the induced subgraph $Q_n[V]$, there is a connected component $G$ containing edges in all $n$ directions.
\end{thm}

Let, then, \mbox{$\mathcal{F}\subseteq 2^{[n]}$} be a simply-rooted family, and let \mbox{$\mathcal{F'} \subseteq \mathcal{F}$} be the family of all sets in $\mathcal{F}$ rooted in exactly one element. Partition $\mathcal{F}'$ to families  $\mathcal{F}'_i := \{A \in \mathcal{F'}| A \text{ is rooted in } i\}$, for $1 \leq i \leq n$. We claim that for $A \in \mathcal{F}'_i$, $B \in \mathcal{F}'_j$ with distinct $i$ and $j$, there can not be an edge (in the hamming cube) between $A$ and $B$. Indeed, assume there is such an edge. That is, $B = A \setminus k$ for some $k \in A$. But notice that by our choice of $A$ and $B$, $A \setminus i, B \setminus j \in \mathcal{G}$. That is, $B \setminus \{i,k\}, B \setminus j \in \mathcal{G}$. Hence, $B=(B \setminus \{i,k\}) \cup B \setminus j \in \mathcal{G}$, a contradiction.\\
\vspace{2.5mm}
Consequentially, every connected component in $Q_n[\mathcal{F}']$ lies entirely, for some $i \in [n]$, in $Q_n[\mathcal{F}'_i]$. Finally, assume by contradiction that $|\mathcal{F}'$ $|>2^{n-1}$. By Theorem \ref{Kotlov}, this means that $Q_n[\mathcal{F}']$ has a connected component with edges in all $n$ directions, and by the previous paragraph, this connected component lies entirely in $Q_n[\mathcal{F}'_i]$ for some $i \in [n]$. But this is impossible, since  $Q_n[\mathcal{F}'_i]$ can not contain edges in the $i$th-direction. Indeed, for every $A \in \mathcal{F}'_i$, $i \in A$.
Thus, $|\mathcal{F}'|\leq 2^{n-1}$, proving the theorem.
\end{proof}

\section{Frankl's conjecture for large families - an improvement}
In this section, we relax slightly the lower-bound on the size of $\mathcal{F}$ in Theorem \ref{abundant}. We formulate the the problem in terms of simply-rooted families, as we have done in section $3$. Formulated thus, Theorem \ref{abundant-2} Says the following:

\begin{thm*} \textbf{\ref{abundant-2}*}
There is some absolute constant $c>0$, such that if  $\mathcal{F} \subseteq 2^{[n]}$ is simply-rooted, and $|\mathcal{F}| \leq (\frac{1}{2}+c)2^n$,then there is some element $i \in [n]$ such that $i$ is rare in $\mathcal{F}$.
\end{thm*}

\begin{proof}
Let $\mathcal{F} \subseteq 2^{[n]}$ be a family fulfilling the assumptions of the theorem. If $|\mathcal{F}|\geq 2^{n-1}$, then we are in the situation of Theorem \ref{abundant}*, and we are done. So let us assume that $|\mathcal{F}|=(\frac{1}{2}+\delta)2^n$, where $0<\delta<c$, $c$ being some absolute constant to be determined later. Assume by contradiction that $\mathcal{F}$ does not have an abundant element. Identify $\mathcal{F}$ with a boolean function \mbox{$f: \{0,1\}^n \to \{-1,1\}$} as in section $2.1$. Observe that for this function $f$:

\begin{equation} \label{zero-level}
\hat{f}(\emptyset)=-2\delta
\end{equation}

\begin{equation} \label{one-level}
\hat{f}(i)=I_i^+(f)-I_i^-(f)>0, \forall i \in [n]
\end{equation}

In a similar fashion to the proof of Theorem \ref{abundant}, we try to bound the total influence of $f$ from below and from above. Let us start with the upper-bound.

Theorem \ref{upper-shadow} provides a bound on the positive influence of $f$. Observe that this theorem implies  $I^+(f)\leq 1$. Indeed, the positive influence of $f$ is exactly the number of simply-rooted elements in $\mathcal{F}$, divided by $2^{n-1}$. We can also lower-bound the difference between the positive and the negative influence of $f$. Equation \ref{one-level} asserts that all the level-$1$ fourier coefficients are positive. Thus:

\begin{equation*}
\sqrt{W^1(f)}=\sqrt{\sum_{i=1}^n \hat{f}(i)^2} < \sum_{i=1}^n \hat{f}(i)=I^+(f)-I^-(f)
\end{equation*}

Combining the equation above with our bound on $I^+(f)$, we can upper bound the total influence:

\begin{equation} \label{up-bound}
I(f)=I^+(f)+I^-(f)<2-\sqrt{W^1(f)}.
\end{equation}

As for the lower-bound of the influence, we in fact provide two lower-bounds, both derived from Corollary \ref{cor} with different values of $k$. For $k=2$, using also equation \ref{up-bound}, we have:

\begin{equation*}
2-\sqrt{W^1(f)}> I^+(f)+I^-(f)\geq 2-W^1(f)-2\hat{f}(\emptyset)^2=2-W^1(f)-8\delta^2
\end{equation*}

After moving terms, this amounts to:

\begin{equation} \label{weight-1}
\sqrt{W^1(f)}(1-\sqrt{W^1(f)})<8\delta^2.
\end{equation}

Taking $k=3$ in Corollary \ref{cor}, and using equation \ref{up-bound} (here we use the weaker version $I(f)\leq 2$) gives:

\begin{equation*}
2\geq I(f)>3-W^2(f)-2W^1(f)-12\delta^2.
\end{equation*}

Or:

\begin{equation} \label{weight-2}
W^2(f)>1-2W^1(f)-12\delta^2.
\end{equation}

We shall presently see that it's impossible for both equations \ref{weight-1} and \ref{weight-2} to hold, thus deriving a contradiction. Assume that equation \ref{weight-1} holds. For $\delta$, and thus $c$, sufficiently small, this equation implies that either $\sqrt{W^1(f)}<9\delta^2$ or $1-\sqrt{W^1(f)}<9\delta^2$.

Assume the latter. Then $\sqrt{W^1(f)}>1-9\delta^2$, hence $W^1(f)>1-18\delta^2$. By Theorem \ref{FKN} (FKN), this means that $f$ is at most $18C_1\delta^2$-distant from some balanced function, $g$. On the other hand, we know that $\hat{f}(\emptyset)=-2\delta$, so $f$ is at least $\delta$-distant from any balanced function. But if $c$, and thus $\delta$, is sufficiently small, then 
$\delta>18C_1\delta^2$, and we arrive at a contradiction.

Assume, then, that $\sqrt{W^1(f)}<9\delta^2$. For $c$ sufficiently small, this implies $W^1(f)<\delta^2$. Plugging this inequality to equation \ref{weight-2}, we learn that:

\begin{equation*}
W^2(f)>1-14\delta^2.
\end{equation*}

Once again, we claim that this is impossible, provided that $c$ is sufficiently small. Indeed, on the one hand, from Theorem \ref{Kindler-Safra} (Kindler-Safra), $f$ is at most $14C_2\delta^2$-distant from some balanced function $g$, and on the other hand $f$ is at least $\delta$-distant from any balanced function, and $\delta>14C_2\delta^2$ if $c$ is small enough.

In any case we arrive at a contradiction, thus proving the theorem.
\end{proof}

\section{Conclusion}
Theorem \ref{upper-shadow}, in the language of boolean functions, says that for any simply-rooted function $f:\{0,1\}^n \to \{-1,1\}$, $I^+(f)\leq 1$. As we have mentioned, this inequality is tight. Consider the following examples: 

\begin{enumerate}
\item $f_1(x)=\chi_{\{1\}}(x)$.
\item $f_2(x)= \chi_{\{1,2\}}(x)$.
\item $f_3(x)= -\frac{1}{2}+\frac{1}{2}\chi_{\{1\}}(x)+\frac{1}{2}\chi_{\{2\}}(x)+\frac{1}{2}\chi_{\{1,2\}}(x)$.
\end{enumerate}

The first two examples are balanced functions. That is, the level-$0$ fourier coefficient is zero. In the last example, this coefficient is $-\frac{1}{2}$. It would be interesting to understand how big can $I^+(f)$ be, as $\hat{f}(\emptyset)$ grows smaller. We have managed to prove, although we do not include the proof here for the sake of brevity,  that if $f:\{0,1\}^n \to \{-1,1\}$ is simply-rooted, and $\hat{f}(\emptyset)<-\frac{1}{2}$, then $I^+(f)<1$. We make the following conjecture regarding the relation of these two qunatities:
\begin{conj}
Let $f:\{0,1\}^n \to \{-1,1\}$ be a simply-rooted function, and let $k \in [0,n-1]$, so that $\hat{f}(\emptyset)\leq -(1-2^{-k})$. Then $I^+(f)\leq (k+1)2^{-k}$.
\end{conj}

The above conjecture, if true, would be tight in the following sense. Let \mbox{$k \in [0,n-1]$}. Then \mbox{$C_k: \{0,1\}^n \to \{-1,1\}$}, defined by  \mbox{$C_k^{-1}(-1)=\{x\in \{0,1\}^n| x_1=1 \vee \dots \vee x_k=1\}$}, is  a simply-rooted function,   \mbox{$\widehat{C_k}(\emptyset)= -(1-2^{-k})$}, and  \mbox{$I^+(C_k)=I(C_k)=(k+1)2^{-k}$}. We do not make a conjecture as to the uniqueness of the function. In fact, notice that, in the examples above, both $f_1$ (which is actually $C_1$) and $f_2$ achieve this bound. \\
\vspace{2.5mm}
This conjecture should be compared with the well-known \textit{edge-isoperimetric inequality} for the Hamming cube, which implies the following:

\begin{thm} \label{edge}
Let $f:\{0,1\}^n \to \{-1,1\}$ be a boolean function, and let \mbox{$k\in [0,n-1]$}, so that \mbox{$-(1-2^{-k})\leq \hat{f}(\emptyset) \leq 0$}. Then $I(f) \geq (k+1)2^{-k}$.
\end{thm}

The very same function $C_k$ shows that this is tight. In other words, we conjecture that, out of all simply-rooted functions $f:\{0,1\}^n \to \{-1,1\}$ with $\hat{f}(\emptyset)=-(1-2^{-k})$, the same function, $C_k$ has both the smallest possible influence and the largest possible \textbf{positive} influence. 

\section{Acknowledgemnts}
I would like to thank Ohad Klein, Nathan Keller and my advisor, Gil Kalai, for many helpful suggestions.


\begin{thebibliography}{22}
\bibitem{Alekseev}
V.B. Alekseev.
On the Number of Intersection Semilattices (in Russian).
\textit{Diskretnaya Matematika}, 1:129-136. 1989.

\bibitem{Balla}
I. Balla, B. Bollob\'{a}s, T. Eccles.
Union-Closed Families of Sets.
\textit{Journal of Combinatorial Theory, Series A}, 120(3):531-544. 2013.

\bibitem{Bosnjak}
 I. Bo\v{s}njak, P. Markovi\'{c}.
 The 11-element Case of Frankl's Conjecture.
\textit{ Electronic Journal of Combinatorics}, 15 . Research Paper 88. 2008,


\bibitem{Bruhn}
H. Bruhn, O. Scahudt.
The Journey of the Union-Closed Sets Conjecture.
\textit{Graphs and Combinatorics},31(6):2043-2074. 2015.

\bibitem{Czedli1}
G. Cz\'{e}dli.
On Averaging Frankl's Conjecture for Large Union-Closed Sets.
\textit{Journal of Combinatorial Theorey, Series A}, 116(3):724-729. 2009.

\bibitem{Eccles}
T. Eccles.
A Stability Result for the Union-Closed Size Problem.
\textit{Combinatorics, Probability and Computing}, 25(3):399-418. 2016.

\bibitem{Falgas-Ravry}
V. Falgas-Ravry.
Minimal Weight in Union-Closed Families. 
\textit{Electronic Journal of Combinatorics}, 19(P95):1–14. 2011.

\bibitem{FKN}
E. Friedgut, G. Kalai, A. Naor.
Boolean Functions Whose Fourier Transform is Concentrated on the First Two Levels.
\textit{Advances in Applied Mathematics}, 29(3):427-437. 2002.

\bibitem{Gao}
W.D. Gao, H.Q. Yu.
 Note on the Union-Closed Sets Conjecture.
\textit{Ars Combinatorica}, 49:280-288. 1998.

\bibitem{Hu}
Y. Hu.
On the Union-Closed Sets Conjecture.
\url{https://arxiv.org/abs/1706.06167}, 2017.

\bibitem{Kindler-Safra}
G. Kindler, S. Safra.
Noise-Resistant Boolean-Functions Are Juntas.
\url{https://pdfs.semanticscholar.org/697b/7ec46680ac2be42e1a27c5a3966d949975f4.pdf}, 2003.


\bibitem{Kleitman}
D.J. Kleitman.
 Extremal Properties of Collections of Subsets Containing No Two Sets
and Their Union.
\textit{Journal of Combinatorial Theory (Series A)}, 20:390–392. 1976.

\bibitem{Knill}
E. Knill
Graph Generated Union-Closed Families of Sets. \url{https://arxiv.org/abs/math/9409215}. 1994.


\bibitem{Kotlov}
A. Kotlov.
Bulky Subraphs of the Hypercube.
\textit{European Journal of Combinatorics}, 21(4):503-507. 2000.

\bibitem{Lo Faro}
 G. Lo Faro.
A Note on the Union-Closed Sets Conjecture.
\textit{Journal of the  Australian Mathematical Society Series A}, 57:230-236. 1994.

\bibitem{Massberg}
J. Ma{\ss}berg.
The Union-Closed Sets Conjecture for Small Families.
\textit{Graphs and Combinatorics}, 32(5): 2047-2051. 2016.

\bibitem{Markovic}
 P. Markovi\'{c}.
 An Attempt at Frankl's Conjecture.
\textit{ Publication de l'institut Mathematique}, 81(95):29-43. 2007.

\bibitem{Morris}
 R. Morris.
 FC-families and Improved Bounds for Frankl's Conjecture.
\textit{European Journal of Combinatorics}, 27 : 269-282. 2006.

\bibitem{O'Donnell}
R. O'Donnell.
\textit{Analysis of Boolean Functions}.
\url{http://www.contrib.andrew.cmu.edu/~ryanod/?cat=63}.

\bibitem{Polymath}
Polymath $11$.
\textit{Gowers's Blog.}
\url{https://gowers.wordpress.com/2016/01/21/frankls-union-closed-conjecture-a-possible-polymath-project/}.


\bibitem{Poonen}
 B. Poonen.
 Union-closed families.
\textit{Jorunal of  Combinatorial Theory Series A}, 59:253-268. 1992.



\bibitem{Reimer}
D. Reimer.
An Average Set Size Theorem.
\textit{Combinatorics, Probability and Computing}, 12(1):89-93. 2003.

\bibitem{Roberts}
 I. Roberts, J. Simpson. 
A Note on the Union-Closed Sets Conjecture.
\textit{Australasian Journal of Cominatorics}, 47:265-267. 2010.

\bibitem{Wojcik}
P. W\'{o}jcik.
Union-Closed Families of Sets.
\textit{Discrete Mathematics}, 199:173-182. 1999.

\bibitem{Vuckovic}
 B. Vu\v{c}kovi\'{c}, M. \v{Z}ivkovi\'{c}. 
 \textit{The 12 Element Case of Frankl's Conjecture}, preprint. 2012.

\end{thebibliography}
\end{document}